\documentclass{amsart}
\usepackage{amsmath}
\usepackage{amssymb}
\usepackage{amsthm}
\usepackage{enumerate}
\usepackage[pdftex]{graphicx}
\usepackage{caption}
\theoremstyle{definition}
\newtheorem{definition}{Definition}[section]
\theoremstyle{plain}
\newtheorem{lemma}[definition]{Lemma}
\newtheorem{theorem}[definition]{Theorem}

\newtheorem{proposition}[definition]{Proposition}
\newtheorem{corollary}[definition]{Corollary}
\theoremstyle{remark}
\newtheorem{remark}[definition]{Remark}

\makeatletter
\@namedef{subjclassname@2020}{
  \textup{2020} Mathematics Subject Classification}
\makeatother

\newcommand{\st}{\operatorname{\textbf{st}}}

\newcommand{\len}{\operatorname{len}}

\begin{document}
\title[Tame extension of almost o-minimal structure]{Tame extension of almost o-minimal structure}
\author[M. Fujita]{Masato Fujita}
\address{Department of Liberal Arts,
Japan Coast Guard Academy,
5-1 Wakaba-cho, Kure, Hiroshima 737-8512, Japan}
\email{fujita.masato.p34@kyoto-u.jp}

\begin{abstract}
We consider an almost o-minimal expansion of an ordered group $\mathcal M=(M,<,+,0,\ldots)$ and its tame extension $\mathcal N=(N,<,+,0,\ldots)$.
We demonstrate that the subset $\{x \in M^n\;|\; \mathcal N \models \Phi(x,a)\}$ of $M^n$ defined by a formula $\Phi(x,y)$ with $\mathcal M$-bounded parameters $a$ in $\mathcal N$ is $\mathcal M$-definable.
We also introduce its corollaries.
\end{abstract}



\maketitle

\section{Introduction}

The notion of tame pairs was initially introduced by Marker and Steinhorn in \cite{MS} (using different terms).
Pillay also considered the same problem in \cite{Pillay}.
It was further developed by van den Dries and Lewenberg in \cite{vdDL, vdD2}.
Using Marker and Steinhorn's result on definable types \cite{MS}, we get that the subset $\{x \in M^n\;|\; \mathcal N \models \Phi(x,a)\}$ of $M^n$ defined by a formula $\Phi(x,a)$ with parameters $a$ in $\mathcal N$ is $\mathcal M$-definable \cite[Theorem 1.1]{vdD2}.
Here, $\mathcal M=(M,<,+,0,\ldots)$ is an o-minimal structure and $\mathcal N=(N,<,+,0,\ldots)$ is its tame extension.
The corollaries of this result together with the outputs of \cite{vdDL} are discussed in \cite{vdD2}.
They are used for the study of limit sets of definable families in \cite{vdD3}.

The author introduced the notion of almost o-minimality in \cite{Fuji}, which satisfies a weaker finiteness condition than o-minimality.
It is a generalization of a locally o-minimal expansion of the set of reals and admits uniform local definable cell decomposition \cite[Theorem 1.7]{Fuji}.
We anticipate that assertions which hold true in o-minimal structures also hold true in almost o-minimal structures under reasonable additional assumptions.
In this paper, we investigate a tame extension of an almost o-minimal structure.
We demonstrate that the subset of $M^n$ defined by a formula with $\mathcal M$-bounded parameters in $\mathcal N$ is $\mathcal M$-definable for an almost o-minimal structure $\mathcal M$ and its tame extension $\mathcal N$.

Let us recall the definitions.

\begin{definition}[\cite{Fuji}]
An expansion $\mathcal M=(M,<,\ldots)$ of densely linearly ordered set without endpoints is \textit{almost o-minimal} if any bounded definable set in $M$ is a finite union of points and open intervals.
\end{definition}

\begin{definition}[\cite{vdDL}]
Let $\mathcal L$ be a language containing a predicate $<$.
Let $\mathcal M=(M,<,\ldots)$ be an $\mathcal L$-expansion of  a dense linear order without endpoints, and $\mathcal N=(N,<, \ldots)$ be its extension.
An element $y \in N$ is \textit{$\mathcal M$-bounded} if $x_1 \leq y \leq x_2$ for some $x_1,x_2 \in M$.
An $\mathcal N$-definable set is called \textit{parameterized by $\mathcal M$-bounded parameters} if there exists a finite subset $A$ of $\mathcal M$-bounded elements in $N$, and it is defined by an $\mathcal L(M \cup A)$-formula.
When the extension $\mathcal M \subseteq \mathcal N$ is elementary, we say that the extension $\mathcal M \subseteq \mathcal N$ is \textit{tame} or $\mathcal M$ is tame in $\mathcal N$ (in \cite{MS}, $\mathcal M$ is called \textit{Dedekind complete} in $\mathcal N$) if for each $\mathcal M$-bounded $y \in N$, there is an $x \in M$ such that either
\begin{enumerate}
\item[(1)] $x=y$, or
\item[(2)] $x<y$ and there are no $x' \in M$ with $x<x'<y$, or
\item[(3)] $y<x$ and there are no $x' \in M$ with $y<x'<x$.
\end{enumerate}
Such an $x$ is uniquely determined by $y$.
We call $x$ the \textit{standard part} of $y$ relative to $\mathcal M$ and write $x=\st_{\mathcal M}(y)$.
We omit the subscript $\mathcal M$ when it is clear from the context.
An $n$-tuple $(y_1, \ldots, y_n) \in N^n$ is \textit{$\mathcal M$-bounded} if each coordinate $y_i$ is $\mathcal M$-bounded.
If $\mathcal M \subseteq \mathcal N$ is a tame extension and  $(y_1, \ldots, y_n) \in N^n$ is $\mathcal M$-bounded, we put $\st_{\mathcal M}(y_1, \ldots, y_n) = (\st_{\mathcal M}(y_1),\ldots,\st_{\mathcal M}(y_n))$.
\end{definition}

Our main result is as follows:

\begin{theorem}\label{thm:main1}
Let $\mathcal L$ be a language.
Let $\mathcal M=(M,<,+,0,\ldots)$ be an almost o-minimal $\mathcal L$-expansion of an ordered group and $\mathcal M \subseteq \mathcal N=(N,<,+,0,\ldots)$ be a tame extension.
Let $a \in N^n$ be an $\mathcal M$-bounded tuple and $\Phi(x,y)$ be $\mathcal L(M)$-formula, where $x$ and $y$ are $m$-tuple and $n$-tuple of free variables, respectively.
Then, the set
$$
S=\{x \in M^m\;|\; \mathcal N \models \Phi(x,a)\}
$$ 
is $\mathcal M$-definable.
In other words, the intersection $\mathcal S \cap M^m$ of an $\mathcal N$-definable subset $\mathcal S$ of $N^m$ parameterized by $\mathcal M$-bounded parameters with $M^m$ is $\mathcal M$-definable.  
\end{theorem}

This paper is organized as follows:
We use several results obtained in previous studies.
We recall them in Section \ref{sec:preliminary}.
Section \ref{sec:proof} is the main body of this paper, and it is devoted to the proof of the theorem.
We introduce the corollaries of the theorem in Section \ref{sec:corollaries}.

In the last of this section, we summarize the terms and notations used in this paper.
When $\mathcal M \subseteq \mathcal N$ is an elementary extension and $S$ is an $\mathcal M$-definable set, the notation $S^{\mathcal N}$ denotes the $\mathcal N$-definable subset defined by the formula defining the $\mathcal M$-definable set $S$.
Since $\mathcal M \subseteq \mathcal N$ is an elementary extension, $S^{\mathcal N}$ is independent of the choice of the formula $\phi$.

When a first-order structure in consideration is clear from the context, the term `definable' means `definable in the structure with parameters.'
We call it $\mathcal M$-definable when we emphasize the structure $\mathcal M$. 
The notation $f|_A$ denotes the restriction of a map $f:X \rightarrow Y$ to a subset $A$ of $X$.
Consider a linearly ordered set without endpoints $(M,<)$.
An open interval is a nonempty set of the form $\{x \in M\;|\; a < x < b\}$ for some $a,b \in M \cup \{\pm \infty\}$.
It is denoted by $(a,b)$ in this paper.
An open box is the Cartesian product of open intervals.
The closed interval is defined similarly and denoted by $[a,b]$.
When an expansion $\mathcal M=(M,<,\ldots)$ of a dense linear order without endpoints is given, the set $M$ equips the order topology induced from the order $<$. 
The space $M^n$ equips the product topology of the order topology.
We consider these topologies.
The space $M^0$ is a singleton with the trivial topology.

\section{Preliminary}\label{sec:preliminary}
We recall the results in \cite{Fuji} and \cite{Fuji4} in this section.
We first recall the definition of dimension of a set definable in a structure.
\begin{definition}[\cite{Fuji4}]
Consider an expansion of a densely linearly order without endpoints $\mathcal M=(M,<,\ldots)$.
Let $X$ be a nonempty definable subset of $M^n$.
The dimension of $X$ is the maximal nonnegative integer $d$ such that $\pi(X)$ has a nonempty interior for some coordinate projection $\pi:M^n \rightarrow M^d$.
We set $\dim(X)=-\infty$ when $X$ is an empty set.
\end{definition}

We also need the following definition:
\begin{definition}[Local monotonicity]
	A function $f$ defined on an open interval $I$ is \textit{locally constant} if, for any $x \in I$, there exists an open interval $J$ such that $x \in J \subseteq I$ and the restriction $f|_J$ of $f$ to $J$ is constant.
	A function $f$ defined on an open interval $I$ is \textit{locally strictly increasing} if, for any $x \in I$, there exists an open interval $J$ such that $x \in J \subseteq I$ and $f$ is strictly increasing on the interval $J$.
	We define a \textit{locally strictly decreasing} function similarly. 
\end{definition}

The author developed the dimension theory for sets definable in a definably complete locally o-minimal structure satisfying the property (a) which is defined in \cite[Definition 1.1]{Fuji4}.
We do not give the definitions of definably complete structures, locally structures and property (a) here.
Their definitions and their references are found in \cite{Fuji4}.
The important fact is that an almost o-minimal expansion of an ordered group is a definably complete locally o-minimal structure satisfying the property (a) thanks to \cite[Corollary 2.12, Lemma 4.6]{Fuji} and \cite[Proposition 2.13]{Fuji4}.
Using this fact, we get the following proposition:

\begin{proposition}\label{prop:dim}
Let $\mathcal M=(M,<,+,0,\ldots)$ be an almost o-minimal expansion of an ordered group.
The following assertions hold true:
\begin{enumerate}
\item[(1)] (Strong local monotonicity) Let $I$ be an interval and $f:I \rightarrow M$ be a definable function.
There exists a mutually disjoint definable partition $I=X_d \cup X_c \cup X_+ \cup X_-$ satisfying the following conditions:
\begin{enumerate}
\item[(i)] the definable set $X_d$ is discrete and closed;
\item[(ii)] the definable set $X_c$ is open and $f$ is locally constant on $X_c$;
\item[(iii)] the definable set $X_+$ is open and $f$ is locally strictly increasing and continuous on $X_+$;
\item[(iv)] the definable set $X_-$ is open and $f$ is locally strictly decreasing and continuous on $X_-$.
\end{enumerate}
\item[(2)] Let $X_1$ and $X_2$ be definable subsets of $M^m$.
Set $X=X_1 \cup X_2$.
Assume that $X$ has a nonempty interior.
At least one of $X_1$ and $X_2$ has a nonempty interior.
\item[(3)] A definable set is of dimension zero if and only if it is discrete.
When it is of dimension zero, it is also closed.
\item[(4)] Let $X$ and $Y$ be definable subsets of $M^n$.
We have 
\begin{align*}
\dim(X \cup Y)=\max\{\dim(X),\dim(Y)\}\text{.}
\end{align*}
\item[(5)] Let $f:X \rightarrow M^n$ be a definable map. 
We have $\dim(f(X)) \leq \dim X$.
\item[(6)] Let $f:X \rightarrow M^n$ be a definable map. 
The notation $\mathcal D(f)$ denotes the set of points at which the map $f$ is discontinuous. 
The inequality $\dim(\mathcal D(f)) < \dim X$ holds true.
\item[(7)] Let $X$ be a definable set.
The notation $\partial X$ denotes the frontier of $X$ defined by $\partial X = \overline{X} \setminus X$.
We have $\dim \partial X < \dim X$.
\item[(8)] Let $\varphi:X \rightarrow Y$ be a definable surjective map whose fibers are equi-dimensional; that is, the dimensions of the fibers $\varphi^{-1}(y)$ are constant.
The equalities $\dim X = \dim Y + \dim \varphi^{-1}(y)$ hold true for all $y \in Y$.  
\item[(9)] Let $X$ be a definable subset of $M^{m+n}$.
The notation $\pi:M^{m+n} \rightarrow M^n$ denotes the projection onto the last $n$ coordinates.
There exists a definable map $\varphi:\pi(X) \rightarrow X$ such that the composition $\pi \circ \varphi$ is the identity map on $\pi(X)$.
\end{enumerate}
\end{proposition}
\begin{proof}
(1) \cite[Theorem 2.11(ii)]{Fuji4};
(2) \cite[Theorem 2.11(iii)]{Fuji4};
(3) \cite[Proposition 3.2]{Fuji4};
(4)-(7) \cite[Theorem 3.8(4)-(7)]{Fuji4};
(8) \cite[Theorem 3.14]{Fuji4};
(9) Well-known.
\end{proof}

The following lemma asserts that an almost o-minimal structure $\mathcal M$ has an o-minimal structure $\mathcal R$ such that any bounded $\mathcal M$-definable set is $\mathcal R$-definable.  
\begin{lemma}\label{lem:included}
Let $\mathcal M=(M,<,+,0,\ldots)$ be an almost o-minimal expansion of an ordered group.
There exists an o-minimal expansion $\mathcal R= (M,<,+,0,\ldots)$ of the ordered group satisfying the following conditions:
\begin{enumerate}
\item[(i)] Any set definable in $\mathcal R$ is $\mathcal M$-definable.
\item[(ii)] Any bounded $\mathcal M$-definable set is definable in $\mathcal R$.
\end{enumerate}
\end{lemma}
\begin{proof}
\cite[Theorem 2.13]{Fuji}
\end{proof}

Almost o-minimal expansions of ordered groups admit partition into multi-cells and uniform local definable cell decomposition.
We first recall the definition of semi-definability.
\begin{definition}
	Let $\mathcal R=(M,<,\ldots)$ be an o-minimal structure.
	A subset $X$ of $M^n$ is \textit{semi-definable in $\mathcal R$} if the intersection $U \cap X$ is definable in $\mathcal R$ for any bounded open box $U$ in $M^n$. 
	
	A semi-definable subset $X$ of $M^n$ is \textit{semi-definably connected} if there are no non-empty proper semi-definable closed and open subsets $Y_1$ and $Y_2$ of $X$ such that $Y_1 \cap Y_2 = \emptyset$ and $X=Y_1 \cup Y_2$.
	For any $x \in X$, there exists a maximal semi-definably connected semi-definable subset $Y$ of $X$ containing the point $x$ by \cite[Theorem 3.6]{Fuji}.
	The set $Y$ is called the semi-definably connected component of $X$ containing the point $x$.
\end{definition}
Let $\mathcal M=(M,<,+,0,\ldots)$ be an almost o-minimal expansion of an ordered group.
Let $\mathcal R$ be the o-minimal structure given in Lemma \ref{lem:included}.
Any set definable in $\mathcal M$ is simultaneously semi-definable in $\mathcal R$.

We next recall the definitions of multi-cells given in \cite[Definition 4.19]{Fuji}.
\begin{definition}
Consider an almost o-minimal expansion of an ordered group $\mathcal M=(M,<,0,+,\ldots)$.
Let $n$ be a positive integer.
A definable subset $X$ of $M^n$ is a \textit{multi-cell} if it satisfies the following conditions:
\begin{itemize}
\item If $n=1$, either $X$ is a discrete definable set or all semi-definably connected components of the definable set $X$ are open intervals. 
\item When $n>1$, let $\pi:M^n \rightarrow M^{n-1}$ be the projection forgetting the last coordinate.
The projection image $\pi(X)$ is a multi-cell and, for any semi-definably connected component $Y$ of $X$, $\pi(Y)$ is a semi-definably connected component of $\pi(X)$ and $Y$ is one of the following forms:
\begin{align*}
Y&=\pi(Y) \times M  \text{,}\\
Y&=\{(x,y) \in \pi(Y) \times M \;|\; y=f(x)\} \text{,}\\
Y &= \{(x,y) \in \pi(Y) \times M \;|\; y>f(x)\} \text{,}\\
Y &= \{(x,y) \in \pi(Y) \times M \;|\; y<g(x)\} \text{ and }\\
Y &= \{(x,y) \in \pi(Y) \times M \;|\; f(x)<y<g(x)\}
\end{align*}
for some semi-definable continuous functions $f$ and $g$ defined on $\pi(Y)$ with $f<g$.
\end{itemize} 
\end{definition}

We obtain the following theorem:
\begin{theorem}\label{thm:multi-cell}
A set definable in an almost o-minimal expansion of an ordered group is partitioned into finitely many multi-cells.
\end{theorem}
\begin{proof}
\cite[Theorem 4.22]{Fuji}
\end{proof}

Let us review the definition of cells.
\begin{definition}[Definable cell decomposition]
	Consider an expansion of dense linear order without endpoints $\mathcal M=(M,<,\ldots)$.
	Let $(i_1, \ldots, i_n)$ be a sequence of zeros and ones of length $n$.
	\textit{$(i_1, \ldots, i_n)$-cells} are definable subsets of $M^n$ defined inductively as follows:
	\begin{itemize}
		\item A $(0)$-cell is a point in $M$ and a $(1)$-cell is an open interval in $M$.
		\item An $(i_1,\ldots,i_n,0)$-cell is the graph of a definable continuous function defined on an $(i_1,\ldots,i_n)$-cell.
		An $(i_1,\ldots,i_n,1)$-cell is a definable set of the form $\{(x,y) \in C \times M\;|\; f(x)<y<g(x)\}$, where $C$ is an $(i_1,\ldots,i_n)$-cell and $f$ and $g$ are definable continuous functions defined on $C$ with $f<g$.
	\end{itemize}
	A \textit{cell} is an $(i_1, \ldots, i_n)$-cell for some sequence $(i_1, \ldots, i_n)$ of zeros and ones.
	The sequence $(i_1, \ldots, i_n)$ is called the \textit{type} of an $(i_1, \ldots, i_n)$-cell.
	An \textit{open cell} is a $(1,1, \ldots, 1)$-cell.
	The dimension of an $(i_1, \ldots, i_n)$-cell is defined by $\sum_{j=1}^n i_j$.
	
	We inductively define a \textit{definable cell decomposition} of an open box $B \subseteq M^n$.
	For $n=1$, a definable cell decomposition of $B$ is a partition $B=\bigcup_{i=1}^m C_i$ into finitely many cells.
	For $n>1$, a definable cell decomposition of $B$ is a partition $B=\bigcup_{i=1}^m C_i$ into finitely many cells such that $\pi(B)=\bigcup_{i=1}^m \pi(C_i)$ is a definable cell decomposition of $\pi(B)$, where $\pi:M^n \rightarrow M^{n-1}$ is the projection forgetting the last coordinate.
	Consider a finite family $\{A_\lambda\}_{\lambda \in \Lambda}$ of definable subsets of $B$.
	A \textit{definable cell decomposition of $B$ partitioning $\{A_\lambda\}_{\lambda \in \Lambda}$} is a definable cell decomposition of $B$ such that the definable sets $A_{\lambda}$ are unions of cells for all $\lambda \in \Lambda$. 
\end{definition}

The following theorem is the uniform local decomposition theorem for almost o-minimal structures.
\begin{theorem}[Uniform local definable cell decomposition]\label{thm:uldcd}
Consider an almost o-minimal expansion of an ordered group $\mathcal M=(M,<,0,+,\ldots)$.
Let $\{A_\lambda\}_{\lambda\in\Lambda}$ be a finite family of definable subsets of $M^{m+n}$.
Take an arbitrary positive element $R \in M$ and set $B=(-R,R)^n$.
Then, there exists a partition into finitely many definable sets 
\begin{equation*}
M^m \times B = X_1 \cup \ldots \cup X_k
\end{equation*}
such that $B=(X_1)_b \cup \ldots \cup (X_k)_b$ is a definable cell decomposition of $B$ for any $b \in M^m$ and either $X_i \cap A_\lambda = \emptyset$ or $X_i \subseteq A_\lambda$ for any $1 \leq i \leq k$ and $\lambda \in \Lambda$.
Furthermore, the type of the cell $(X_i)_b$ is independent of the choice of $b$ with $(X_i)_b \not= \emptyset$.
Here, the notation $S_b$ denotes the fiber of a definable subset $S$ of $M^{m+n}$ at $b \in M^m$.
\end{theorem}
\begin{proof}
\cite[Theorem 1.7]{Fuji}
\end{proof}

\section{Proof of the main theorem}\label{sec:proof}
This section is devoted to the proof of Theorem \ref{thm:main1}.
We prove it and the following lemma simultaneously.
Our proof is partially inspired by the geometric proof of a similar assertion for o-minimal structures in \cite{Pillay}.

\begin{lemma}\label{lem:main1}
Let $\mathcal L$, $\mathcal M$, $\mathcal N$ and $a=(a_1,\ldots, a_n) \in N^n$ be as in Theorem \ref{thm:main1}.
Assume that $n>0$ and there are no $\mathcal L(\mathcal M)$-formula $\phi(y)$ such that $\mathcal N \models \phi(a)$ and $\dim(\{y \in M^n\;|\; \mathcal M \models \phi(y)\})<n$.
Set $\alpha =(a_1, \ldots, a_{n-1}) \in N^{n-1}$ and $\beta = a_n \in N$.

Let $\pi:M^{m+n-1} \rightarrow M^m$ be the coordinate projection onto the first $m$-coordinates.
We consider $\mathcal M$-definable subsets $T$ and $C$ of $M^m$ and $M^{m+n-1}$, respectively, with $T=\pi(C)$.
Let $f:C \rightarrow M$ be a bounded $\mathcal M$-definable function.
Then, the set $$\mathfrak T(T,C,f)=\{t \in T\;|\; \alpha \in (C^{\mathcal N})_t,\ f(t,\alpha)<\beta\}$$ is $\mathcal M$-definable, where $(C^{\mathcal N})_t$ denotes the fiber of the $\mathcal N$-definable set $C^{\mathcal N}$ at $t$.
\end{lemma}

\begin{proof}[Proof of Theorem \ref{thm:main1} and Lemma \ref{lem:main1}]
By Lemma \ref{lem:included}, there exists an o-minimal structure $\mathcal R$ having the same universe $M$ as $\mathcal M$ such that (i) any $\mathcal R$-definable set is definable in $\mathcal M$ and (ii) any bounded $\mathcal M$-definable set is definable in $\mathcal R$.
We fix such an o-minimal structure $\mathcal R$ in the proof.

\medskip 
In the proof of Theorem \ref{thm:main1}, we set 
\begin{align*}
&X=\{(x,y) \in M^m \times M^n\;|\; \mathcal M \models \Phi(x,y)\} \text{ and }\\
&T=\{x \in M^m\;|\; \mathcal M \models \exists y\  \Phi(x,y)\}.
\end{align*}
Let $C$ be the image of $X$ under the coordinate projection of $M^{m+n}$ forgetting the last coordinate in the proof of Theorem \ref{thm:main1}.
The same symbols $C$ and $T$ are used in the statement of the lemma, but this abuse of symbols will not confuse the readers.

Let $\len(a)$ denote the length of the tuple $a$.
We prove Theorem \ref{thm:main1} and Lemma \ref{lem:main1} by induction on $(\len(a),\dim T)$ in the lexicographic order simultaneously.
Theorem \ref{thm:main1} obviously holds true when $n=0$.
In the rest of the proof, we assume that $n>0$.
We set $\alpha =(a_1, \ldots, a_{n-1}) \in N^{n-1}$ and $\beta = a_n \in N$.
We first show that Lemma \ref{lem:main1} implies Theorem \ref{thm:main1}.
\medskip

\textbf{Claim 1.} We may assume that $X_t$ is a bounded cell for any $t \in T$.
\medskip

Set $b=\st(a)$ and fix a bounded open box $B$ in $M^n$ containing the point $b$. 
There is a partition
$$X \cap (M^m \times B)=X_1 \cup \ldots \cup X_k$$
such that either the fiber $(X_i)_t$ of $X_i$ at $t$ is empty or a cell for any $1 \leq i \leq k$ and $t \in T$ by Theorem \ref{thm:uldcd}.
Set $S_i=\{ x \in M^m\;|\; (x,a) \in X_i^\mathcal N\}$.
Since $S=\bigcup_{i=1}^kS_i$, the set $S$ is $\mathcal M$-definable if $S_i$ are $\mathcal M$-definable for all $1 \leq i \leq k$.
Considering $S_i$ instead of $S$, we may assume that $X_t$ is a bounded cell for any $t \in T$.

\medskip

\textbf{Claim 2.} We may assume that there are no $\mathcal L(M)$-formula $\phi(y)$ such that $\mathcal N \models \phi(a)$ and $\dim(\{y \in M^n\;|\; \mathcal M \models \phi(y)\})<n$.
\medskip

Assume that such a formula $\phi(y)$ exists.
Set $V=\{y \in M^n\;|\; \mathcal M \models \phi(y)\}$.
We may assume that $V$ is bounded considering $V \cap B$ instead of $V$ if necessary.
The $\mathcal M$-definable set $V$ is definable in $\mathcal R$.
Apply the definable cell decomposition theorem for o-minimal structures \cite{vdD}.
We can get a cell decomposition $\{C_i\}_{i=1}^l$ partitioning $V$. 
Note that the $\mathcal R$-definable sets $C_i$ are also definable in $\mathcal M$.
We have $a \in C_j^{\mathcal N}$ for some $1 \leq j \leq k$.
We may assume that $V$ is a bounded cell definable in $\mathcal R$ considering the cell $C_j$ instead of $V$.
Set $\ell=\dim(V)$.
Let $p:M^n \rightarrow M^\ell$ be the projection onto the first $\ell$ coordinates.
Since $V$ is a cell, we may assume that $p(V)$ is open and $V$ is the graph of an $\mathcal R$-definable continuous map $\varphi:p(V) \rightarrow M^{n-\ell}$ by permuting the coordinates if necessary.
Consider the $\mathcal L(M)$-formula $\psi(x,z)=(z \in p(V)) \wedge \Phi(x,(z,\varphi(z)))$.
We obviously have $S=\{x \in M^m\;|\; \mathcal N \models \psi(x,(a_1,\ldots, a_\ell))\}$.
When $\ell<n$, the set $S$ is $\mathcal M$-definable by the induction hypothesis.
We have succeeded in the reduction. 
\medskip

We are now ready to demonstrate that Lemma \ref{lem:main1} implies Theorem \ref{thm:main1}.
By Claim 1 and the definition of cells, there exist bounded $\mathcal M$-definable functions $f,g:C \rightarrow M$ such that, $f<g$ on their domain, the functions $f(t,\cdot)$ and $g(t,\cdot)$ defined on $C_t$ are continuous for any $t \in T$, and we have either 
\begin{align*}
&X=\{(t,x,y) \in T \times M^{n-1} \times M\;|\; x \in C_t, \ f(t,x)<y<g(t,x)\} \text{ or }\\
&X=\{(t,x,y) \in T \times M^{n-1} \times M\;|\; x \in C_t, \ f(t,x)=y\},
\end{align*}
where $C_t$ denotes the fiber of $C$ at $t$.
In the second case, the set $S$ is an empty set.
In fact, if $S$ is not empty, there exists $t \in M^n$ with $\beta=f(t,\alpha)$, which contradicts Claim 2.
By the definition of $S$ and $X$, we have
\begin{align*}
S &=\{t \in M^m\;|\; t \in T^{\mathcal N},\ \alpha \in (C^\mathcal N)_t,\ f(t,\alpha)<\beta<g(t,\alpha)\}\\
&=\{t \in T\;|\; \alpha \in (C^\mathcal N)_t,\ f(t,\alpha)<\beta<g(t,\alpha)\}
\end{align*}
because $\mathcal N$ is an elementary extension of $\mathcal M$.
We therefore have
\begin{align*}
S&=\{t \in T\;|\; \alpha \in (C^\mathcal N)_t,\ f(t,\alpha)<\beta\} \\
&\qquad \setminus (\{t \in T\;|\; \alpha \in (C^\mathcal N)_t,\ g(t,\alpha)<\beta\} \cup \{t \in T\;|\; \alpha \in (C^\mathcal N)_t,\ g(t,\alpha)=\beta\}).
\end{align*}
Note that Claim 2 implies that the tuple $a$ satisfies the assumption in Lemma \ref{lem:main1}.
The first and second sets in the right hand of the equality are $\mathcal M$-definable by Lemma \ref{lem:main1}.
The third set is an empty set; otherwise, it contradicts Claim 2.
We have demonstrated that Lemma \ref{lem:main1} implies Theorem \ref{thm:main1}.
\medskip

The remaining task is to demonstrate Lemma \ref{lem:main1}.
The induction hypothesis implies that we may assume that $\mathfrak T(T,C,f)$ is of a simpler form.
\medskip

\textbf{Claim 3.} We may assume that $\alpha \in (C_{\mathcal N})_t$ for any $t \in T$ and 
\begin{equation}
\mathfrak T(T,C,f)=\{t \in T\;|\;f(t,\alpha)<\beta\}. \label{eq:aaa1}
\end{equation}
\medskip

In fact, by the induction hypothesis, Theorem \ref{thm:main1} holds for $\alpha$. 
Therefore, the set 
\begin{align*}
D&=\{t \in T\;|\; \alpha \in (C^{\mathcal N})_t\}=\{t \in M^m\;|\; t \in T^{\mathcal N},\ \alpha \in (C^{\mathcal N})_t\}
\end{align*}
is an $\mathcal M$-definable set.
We obviously have $\mathfrak T(T,C,f)=\{t \in D\;|\; f(t,\alpha)<\beta\}=\mathfrak T(D, C \cap (D \times M^{n-1}), f|_{C \cap (D \times M^{n-1})})$, where the notation $f|_{C \cap (D \times M^{n-1})}$ denotes the restriction of $f$ to $C \cap (D \times M^{n-1})$.
The tuple $(D, C \cap (D \times M^{n-1}), f|_{C \cap (D \times M^{n-1})})$ satisfies the equality (\ref{eq:aaa1}).
Therefore, we may assume that the conditions in Claim 3 hold true, considering the tuple $(D, C \cap (D \times M^{n-1}), f|_{C \cap (D \times M^{n-1})})$ instead of the tuple $(T,C,f)$.
\medskip

Set $d=\dim T$.
We are now ready to show Lemma \ref{lem:main1} when $d=0$.
We reduce to the case in which $f$ is defined on $T \times M^{m-1}$.
In fact, take an element $c \in M$ larger than $\beta$.
It is possible because $\beta$ is $\mathcal M$-bounded.
Consider the $\mathcal M$-definable function $f': T \times M^{n-1} \rightarrow M$ which coincides with $f$ on $C$ and is constantly $c$ elsewhere.
We have 
\begin{align*}
\mathfrak T(T,C,f)&=\{t \in T\;|\; f(t,\alpha)<\beta\} =\{t \in T\;|\; f'(t,\alpha)<\beta\} \\
&= \mathfrak T(T,T \times M^{n-1},f').
\end{align*}
Replacing $f$ with $f'$, we may assume that $f$ is defined on $T \times M^{n-1}$. 

Consider the $\mathcal M$-definable set $$E=\{(x,y) \in M^{n-1} \times M\;|\; \mathcal M \models \exists t \in T\ f(t,x)<y\}.$$
When $a=(\alpha,\beta) \not\in E^{\mathcal N}$, the set $\mathfrak T(T,C,f)$ is an empty set, and it is obviously $\mathcal M$-definable.
Hence, we may assume that $a=(\alpha,\beta) \in E^{\mathcal N}$.
By the assumption of the lemma, the definable set $E$ has a nonempty interior.
Consider the map $\lambda: E \rightarrow M$ defined by $$\lambda(x,y)=\sup\{f(t,x)\;|\;t \in T,\ f(t,x)<y\}.$$
We also consider the $\mathcal M$-definable set given by $$Z=\{(t,x,y) \in T \times M^{n-1} \times M\;|\; f(t,x)<y,\ f(t,x)=\lambda(x,y)\}.$$
The set $\{f(t,x)\;|\;t \in T,\ f(t,x)<y\}$ is of dimension zero for any fixed $(x,y) \in E$ by Proposition \ref{prop:dim}(5) because $T$ is of dimension zero.
It is closed and discrete by Proposition \ref{prop:dim}(3).
It means that, for any $(x,y) \in E$, the fiber of the set $Z$ at $(x,y)$ given by
$$Z_{(x,y)}=\{t \in T\;|\;\ f(t,x)<y,\ f(t,x)=\lambda(x,y)\}$$ 
is not an empty set.
It immediately implies the inequality 
\begin{equation}
\lambda(x,y)<y. \label{eqeq:a1}
\end{equation}

We can construct an $\mathcal M$-definable map $\tau:E \rightarrow T$ so that $(\tau(x,y),x,y) \in Z$ by Proposition \ref{prop:dim}(9).

The set of points at which $\tau$ is discontinuous is of dimension smaller than $\dim E=n$ by Proposition \ref{prop:dim}(6).
In particular, it has an empty interior.
Therefore, the interior $U$ of the set 
$$\{x \in E\;|\; \tau \text{ is continuous at } x\}$$
has a nonempty $\mathcal M$-definable subset of $M^n$ with $\dim (E \setminus U)<n$ by Proposition \ref{prop:dim}(2) and (7).
By the assumption on the tuple $a$ in the lemma, we get $a=(\alpha,\beta) \in U^{\mathcal N}$.

Take a bounded open box $B$ in $M^n$ containing the point $b=\st(a)$.
We immediately get $a \in B^{\mathcal N}$.
The intersection $U \cap B$ is definable in $\mathcal R$.
Apply the definable cell decomposition theorem for o-minimal structures \cite[Chapter 3, Theorem 2.11]{vdD}.
We can partition $U \cap B$ into $\mathcal R$-definable cells.
Let $U \cap B=U_1 \cup \ldots \cup U_l$ be such a partition.
Note that $U_i$ are all $\mathcal M$-definable for all $1 \leq i \leq l$.
Since $a \in (U \cap B)^{\mathcal N}$, we have $a \in (U_i)^{\mathcal N}$ for some $ 1 \leq i \leq l$.
We may assume that $i=1$ without loss of generality.
Recall that the map $\tau$ is continuous and $U_1$ is definably connected because $U_1$ is a cell.
Therefore, the image $\tau(U_1)$ is definably connected.
Since $\tau(U_1)$ is a subset of the discrete set $T$, the set $\tau(U_1)$ is a singleton.
Let $u \in M^m$ be the unique point contained in $\tau(U_1)$.
We have 
\begin{equation}
f(u,x)=\lambda(x,y) \label{eqeq:a2}
\end{equation}
for all $(x,y) \in U_1$ because $\tau(x,y) \in Z_{(x,y)}$.

We want to show that
\begin{equation}
\mathcal M \models \forall t \in T\ (f(t,x)<y\ \leftrightarrow \ f(t,x) \leq f(u,x)) \label{eqeq:a3}
\end{equation}
for all $(x,y) \in U_1$.
Fix arbitrary $(x,y) \in U_1$ and $t \in T$.
Assume first that $f(t,x)<y$.
We have $f(t,x) \leq \lambda(x,y)$ by the definition of the function $\lambda$. 
We then get $f(t,x) \leq f(u,x)$ by the equality (\ref{eqeq:a2}).
We next prove the opposite implication.
Assume that $f(t,x) \leq f(u,x)$.
The inequality (\ref{eqeq:a1}) and the equality (\ref{eqeq:a2}) immediately imply the inequality $f(t,x)<y$.
We have shown the equivalence (\ref{eqeq:a3}).
We now get 
$$\mathcal N \models \forall t \in T^{\mathcal N},\ (f(t,\alpha)<\beta\ \leftrightarrow \ f(t,\alpha) \leq f(u,\alpha))$$
because $a=(\alpha,\beta) \in (U_1)^{\mathcal N}$.
In particular, we obtain $$\mathfrak T(T,C,f)=\{t \in T\;|\; f(t,\alpha) \leq f(u,\alpha)\}.$$
The right hand of the equality is $\mathcal M$-definable by the induction hypothesis.
We have demonstrated Lemma \ref{lem:main1} when $d=0$.
\medskip

We finally prove Lemma \ref{lem:main1} for $d>0$.
Let $\mathfrak p:M^m \rightarrow M^{m-1}$ be the coordinate projection forgetting the last coordinate.
We reduce to a simpler case.
\medskip

\textbf{Claim 4.} We may assume the following:
\begin{itemize}
\item For any $u \in \mathfrak p(T)$ and $w \in M^{n-1}$, the set $I_{u,w}(C)=\{v \in M\;|\; (u,v,w) \in C\}$ is either an empty set or the union of open intervals, and the map $f_{u,w}: I_{u,w}(C) \rightarrow M$ given by $f_{u,w}(v)=f(u,v,w)$ is locally strictly increasing and continuous;
\item We have $\alpha \in (C^{\mathcal N})_t$ for any $t \in T$.
\end{itemize}
\medskip

We prove Claim 4.
We first reduce to the case in which $(*)$ the fiber $T_u=\{v \in M\;|\; (u,v) \in T\}$ is of dimension one for any $u \in \mathfrak p(T)$.
Let $\pi_i:M^m \rightarrow M$ be the coordinate projection onto the $i$-th coordinate.
Set $d'=\min\{ 1 \leq i \leq m\;|\; \exists u \ \dim \pi_i^{-1}(u) \cap T=1\}$.
We prove it by reverse induction on $d'$.
When $d'<m$, permute the $i$-th and $m$-th coordinate.
We can reduce to the case in which $d'=m$.
We set $P_i=\{u \in \mathfrak p(T)\;|\; \dim T_u=i\}$ for $i=0,1$.
Set $T_i=\mathfrak p^{-1}(P_i) \cap T$. 
The set $T_1$ satisfies the condition $(*)$.
When $\dim T_0<d$, the lemma is true for $T_0$ by the induction hypothesis on $d$.
We can reduce to the case in which the condition $(*)$ is satisfied when $\dim T_0=d$ by the induction hypothesis on $d'$.
We have $\dim \mathfrak p(T)=d-1$ by  Proposition \ref{prop:dim}(8) when the condition $(*)$ is satisfied.

We next reduce to the case in which $T$ is a multi-cell and $\dim \mathfrak p(T)=d-1$.
We may assume that the fiber $T_u=\{v \in M\;|\; (u,v) \in T\}$ is of dimension one for some $u \in \mathfrak p(T)$ by permuting the coordinates if necessary.
Apply Theorem \ref{thm:multi-cell}.
Let $T=\bigcup_{i=1}^l D_i$ be a partition into multi-cells.
Since we have $\mathfrak T(T,C,f)=\bigcup_{i=1}^l \mathfrak T(D_i, C \cap (D_i \times M^{n-1}), f|_{C \cap (D_i \times M^{n-1})})$, the set $\mathfrak T(T,C,f)$ is $\mathcal M$-definable if $\mathfrak T(D_i, C \cap (D_i \times M^{n-1}), f|_{C \cap (D_i \times M^{n-1})})$ are $\mathcal M$-definable for all $1 \leq i \leq l$.
Therefore, we may assume that $T$ is a multi-cell without loss of generality.

For simplicity, we assume that $C=T \times M^{n-1}$.
In fact, the same proof as the case in which $d=0$ justifies this assumption. 
We consider the sets
\begin{align*}
B_1&=\{(u,v,w) \in T \times M^{n-1}\;|\; \exists s_1, s_2,\ s_1<v<s_2,\ \{u\} \times (s_1,s_2) \times \{w\} \subseteq C,\\
&\ f_{u,w} \text{ is strictly increasing and continuous on }(s_1,s_2)\},\\
B_2&=\{(u,v,w) \in T \times M^{n-1}\;|\; \exists s_1, s_2,\ s_1<v<s_2,\ \{u\} \times (s_1,s_2) \times \{w\} \subseteq C,\\
&\ f_{u,w} \text{ is strictly decreasing and continuous on }(s_1,s_2)\}\text{ and }\\
B_3&=\{(u,v,w) \in T \times M^{n-1}\;|\; \exists s_1, s_2,\ s_1<v<s_2,\ \{u\} \times (s_1,s_2) \times \{w\} \subseteq C,\\
&\ f_{u,w} \text{ is constant on }(s_1,s_2)\},
\end{align*}
where $u$, $v$ and $w$ are elements in $M^{m-1}$, $M$ and $M^{n-1}$, respectively.
Set $B_4=(T \times M^{n-1}) \setminus (B_1 \cup B_2 \cup B_3)$.
By Proposition \ref{prop:dim}(1), the set $\{v \in M\;|\; (u,v,w) \in B_4\}$ is of dimension not greater than zero for any fixed $u \in \mathfrak p(T)$ and $w \in M^{n-1}$. 
We get $\dim B_4 < d+n-1$ by Proposition \ref{prop:dim}(8).
We also have $$\mathfrak T(T,C,f)=\{t \in T\;|\; f(t,\alpha)<\beta\}=\bigcup_{i=1}^4\{t \in T\;|\;\alpha \in (B_i^\mathcal N)_t,\  f(t,\alpha)<\beta\}.$$

Set 
$T_i=\{t \in T\;|\; \alpha \in (B_i^{\mathcal N})_t\}$
for $1 \leq i \leq 4$.
The sets $T_i$ are $\mathcal M$-definable by the induction hypothesis.
By the definition of $T_i$, the condition that $\alpha \in (B_i^{\mathcal N})_t$ is satisfied for any $t \in T_i$.
We also set $C_i=(T_i \times M^{n-1}) \cap B_i$ for all $1 \leq i \leq 4$.
For any $t \in T_i$, we have $(B_i)_t=(C_i)_t$ and we get $(B_i^{\mathcal N})_t=(C_i^{\mathcal N})_t$.
Therefore, we have $\alpha \in (C_i^{\mathcal N})_t$ for any $t \in T_i$.
We get $$\{t \in T\;|\;\alpha \in (B_i^\mathcal N)_t,\  f(t,\alpha)<\beta\}=\{t \in T_i\;|\;\   f(t,\alpha)<\beta\}=\mathfrak T(T_i, C_i, f|_{C_i}).$$
We have shown $\mathfrak T(T,C,f)=\bigcup_{i=1}^4 \mathfrak T(T_i, C_i, f|_{C_i})$.
Therefore, we have only to prove that $\mathfrak T(T_i, C_i, f|_{C_i})$ is $\mathcal M$-definable for each $1 \leq i \leq 4$.
As for the case in which $i=4$, we have $\dim (T_4)<d$ by Proposition \ref{prop:dim}(8) because the fiber $(B_4)_t$ is of dimension $n-1$ for each $t \in T_4$ because $(B_4^{\mathcal N})_t$ contains the point $\alpha$.
The set $\mathfrak T(T_4, C_4, f|_{C_4})$ is $\mathcal M$-definable by the induction hypothesis.
We consider the case in which $i=1$.
The tuple $(T_1,C_1 , f|_{C_1 })$ satisfies the conditions in the claim.
The set $\mathfrak T(T_1, C_1, f|_{C_1})$ is also $\mathcal M$-definable by the assumption of the claim.
The case in which $i=2$ is also easy.
Set 
\begin{align*}
&\widehat{T_2}=\{(u,v) \in M^{m-1} \times M\;|\; (u,-v) \in T_2\} \text{ and }\\
&\widehat{C_2}=\{(u,v,w) \in M^{m-1} \times M \times M^{n-1}\;|\; (u,-v,w) \in C_2\}
\end{align*}
The $\mathcal M$-definable function $\widehat{f}:\widehat{C_2} \rightarrow M$ is given by $\widehat{f}(u,v,w)=f(u,-v,w)$.
The tuple $(\widehat{T_2},\widehat{C_2},\widehat{f})$ satisfies the assumption of the claim.
The set $\mathfrak T(\widehat{T_2},\widehat{C_2},\widehat{f})$ is $\mathcal M$-definable.
The set $\mathfrak T(T_2,C_2,f|_{C_2})$ is also $\mathcal M$-definable because $\mathfrak T(T_2,C_2,f|_{C_2})=\{(u,v) \in M^{n-1} \times M\;|\;(u,-v) \in \mathfrak T(\widehat{T_2},\widehat{C_2},\widehat{f})\}$.

We finally consider the case in which $i=3$.
We set $T=T_3$, $C=C_3$ and $f=f|_{C_3}$ for the simplicity of notations.
The following conditions are satisfied:
\begin{itemize}
\item For any $u \in \mathfrak p(T)$ and $w \in M^{n-1}$, the set $I_{u,w}(C)$ is either an empty set or the union of open intervals, and the map $f_{u,w}: I_{u,w}(C) \rightarrow M$ is locally constant;
\item We have $\alpha \in (C^{\mathcal N})_t$ for any $t \in T$.
\end{itemize}
We have only to demonstrate that $\mathfrak T(T,C,f)$ is $\mathcal M$-definable in this case.
We construct an $\mathcal M$-definable map $\rho:T \rightarrow T$ as follows:
Fix an arbitrary element $u \in \mathfrak p(T)$ and $v \in M$ with $(u,v) \in T$.
We define $r_1(u,v)$ and $r_2(u,v)$ by
\begin{align*}
& r_1(u,v) = \inf \{v' \in M\;|\; \forall v'',\ (v'<v''<v) \rightarrow (u,v'') \in T\} \text{ and }\\
& r_2(u,v) = \sup \{v' \in M\;|\; \forall v'',\ (v<v''<v') \rightarrow (u,v'') \in T\}.
\end{align*}
Take a positive element $c \in M$.
We set $\rho(u,v)$ as follows:
\begin{align*}
&\rho(u,v)=\left\{\begin{array}{ll}
(u,0) & \text{if }r_1(u,v)=-\infty \text{ and } r_2(u,v)=+\infty,\\
(u,r_2(u,v)-c) & \text{if } r_1(u,v)=-\infty \text{ and }r_2(u,v)<\infty,\\
(u,r_1(u,v)+c) &  \text{if } r_1(u,v)>-\infty \text{ and }r_2(u,v)=\infty,\\
(u,(r_1(u,v)+r_2(u,v))/2) & \text{otherwise.}
\end{array}\right.
\end{align*}
It is easy to demonstrate that, for any $u \in \mathfrak p(T)$, the fiber $\rho(T)_u$ of $\rho(T)$ at $u$ does not contain an interval.
It means that $\dim \rho(T)_u = 0$.
We get $\rho(T)<d$ by Proposition \ref{prop:dim}(8) because $\dim \mathfrak p(T)=d-1$.
The restriction of $\rho$ to $\rho(T)$ is an identity map by the definition.
A locally constant function definable in a definably complete structure defined on an open interval is constant.
Therefore, the restriction of $f_{u,v}$ to an open interval is constant.
It implies the equality $$f(u,v,w)=f(\rho(u,v),w)$$ for any $(u,v,w) \in C$.
In particular, we obtain
$$f(u,v,\alpha)=f(\rho(u,v),\alpha)$$
for any $(u,v) \in T$.
Set $C'=(\rho(T) \times M^{n-1}) \cap C$.
The map $f'$ is defined as the restriction of $f$ to $C'$.
The set $\mathfrak T(\rho(T),C',f')$ is an $\mathcal M$-definable set by the induction hypothesis because $\dim \rho(T)<d$.
On the other hand, we get
\begin{align*}
\mathfrak T(T,C,F) &= \{(u,v) \in T\;|\; f(u,v,\alpha)<\beta\} = \{(u,v) \in T\;|\; f(\rho(u,v),\alpha)<\beta\}\\
&=\rho^{-1}(\{(u,v) \in \rho(T)\;|\; f(u,v,\alpha)<\beta\})=\rho^{-1}(\mathfrak T(\rho(T),C',f')).
\end{align*} 
It implies that the set $\mathfrak T(T,C,F)$ is an $\mathcal M$-definable set.
We have completed the proof for the case in which $i=3$, and we also have demonstrated Claim 4.
\medskip

The maps $\rho:T \rightarrow T$, $r_1:T \rightarrow M \cup \{-\infty\}$ and $r_2:T \rightarrow M \cup \{+\infty\}$ are the map defined in the proof of Claim 4.
Set $T'=\rho(T)$.
The maps $\kappa_1: T' \rightarrow M \cup \{-\infty\}$ and $\kappa_2: T' \rightarrow M \cup \{\infty\}$ are the restrictions of $r_1$ and $r_2$ to $T'$, respectively.
We use these notations in the rest of the proof.
The first condition of Claim 4 is expressed by a first-order formula.
Therefore, the univariate function $f(u, \cdot, \alpha)$ is also strictly increasing and continuous in the interval $(r_1(u,v),r_2(u,v))$ for all $(u,v) \in T$.
We use this fact without notice.
\medskip

\textbf{Claim 5.} We may further assume that, for any $u \in \mathfrak p(T)$ and any maximal interval $I$ contained in the fiber $T_u$ of $T$ at $u$, there exists a unique $d \in I^\mathcal N$ such that $$f(u,d,\alpha)=\beta.$$
\medskip

We demonstrate Claim 5.
Since $\mathcal N$ is an elementary extension of $\mathcal M$,  $f(u, \cdot,\alpha)$ is strictly increasing in $I^{\mathcal N}$ by Claim 4.
Therefore, there is at most one element $d \in I^{\mathcal N}$ satisfying the condition in Claim 5.
Set $$W_1=\{(u,v) \in T'\;|\; \forall v' \in N,\ (\kappa_1(u,v)<v'<\kappa_2(u,v)) \rightarrow (f(u,v',\alpha)<\beta)\}.$$
Theorem \ref{thm:main1} holds true for $T'$ by the induction hypothesis.
Therefore, the set $W_1$ is $\mathcal M$-definable.
Consider the set $$\widetilde{W_1}=\{(u,v) \in T\;|\; \forall v' \in N,\ (r_1(u,v)<v'<r_2(u,v)) \rightarrow (f(u,v',\alpha)<\beta)\}.$$
We have $\widetilde{W_1}=\rho^{-1}(W_1)$.
It implies that $\widetilde{W_1}$ is also $\mathcal M$-definable.
We put $$\widetilde{W_2}=\{(u,v) \in T\;|\; \forall v' \in N,\ (r_1(u,v)<v'<r_2(u,v)) \rightarrow (f(u,v',\alpha)>\beta)\}.$$
It is also $\mathcal M$-definable similarly.

Set $\widetilde{T}=T \setminus (\widetilde{W_1} \cup \widetilde{W_2})$, $\widetilde{C}=C \cap (\widetilde{T} \times M^{n-1})$ and $\widetilde{f}=f|_{C'}$.
It is obvious that $\mathfrak T(T,C,f)=\widetilde{W_1} \cup \mathfrak T(\widetilde{T},\widetilde{C},\widetilde{f})$.
We may assume that the condition in Claim 5 is satisfied considering the tuple $(\widetilde{T},\widetilde{C},\widetilde{f})$ instead of the tuple $(T,C,f)$.
We have proven Claim 5.
\medskip

Thanks to Claim 5, for any $(u,v) \in T'$, we can find the unique $d \in N$ satisfying $\kappa_1(u,v)<d<\kappa_2(u,v)$ and $f(u,d,\alpha)=\beta$.
We denote such $d$ by $\delta(u,v)$.
It induces a map $\delta:T' \rightarrow N$.  
Consider the $\mathcal L(M)$-formula:
\begin{align}
\Theta(u,v,v',w,z) &= ((u,v) \in T') \wedge ((u,v,w,z) \in C) \wedge (\kappa_1(u,v)<v'<\kappa_2(u,v))\nonumber\\
& \qquad \wedge (f(u,v',w)=z).\label{eqeq:theta}
\end{align}
Note that the graph of $\delta$ is given by $\{(u,v,v') \in T' \times N\;|\; \mathcal N \models \Theta(u,v,v',\alpha,\beta)\}$.

We consider the following two $\mathcal L(M \cup \{\alpha\})$-formulas:
\begin{align*}
\widetilde{\epsilon_1}(u_1,v_1,v',u_2,v_2) &= ((u_1,v_1) \in T) \wedge ((u_2,v_2) \in T) \wedge (r_1(u_1,v_1)<v' < r_2(u_1,v_1))\\
& \wedge (f(u_2,v_2,\alpha) \geq f(u_1,v',\alpha)),\\
\widetilde{\epsilon_2}(u_1,v_1,v',u_2,v_2) &= ((u_1,v_1) \in T) \wedge ((u_2,v_2) \in T) \wedge (r_1(u_1,v_1) < v' < r_2(u_1,v_1))\\
& \wedge (f(u_2,v_2,\alpha) \leq f(u_1,v',\alpha)).
\end{align*}
There is an $\mathcal L(M)$-formula $\epsilon'_i(u_1,v_1,v',u_2,v_2)$ such that
\begin{align*}
\mathcal N \models \widetilde{\epsilon_i}(u_1,v_1,v',u_2,v_2) \Leftrightarrow \mathcal M \models \epsilon'_i(u_1,v_1,v',u_2,v_2)
\end{align*}
for any $i=1,2$ and $(u_1,v_1,v',u_2,v_2) \in T \times M \times T$ because Theorem \ref{thm:main1} holds true for $\alpha$ by the induction hypothesis.
We set 
\begin{align*}
\epsilon_1(u_1,v_1,u_2,v_2) &= \forall v'\ (r_1(u_1,v_1)<v'<r_2(u_1,v_1) \wedge v'<v_1)\\
&\qquad
\rightarrow \epsilon'_1(u_1,v_1,v',u_2,v_2),\\
\epsilon_2(u_1,v_1,u_2,v_2) &= \forall v'\ (r_1(u_1,v_1)<v'<r_2(u_1,v_1) \wedge v'>v_1)\\
&\qquad 
\rightarrow \epsilon'_2(u_1,v_1,v',u_2,v_2) \text{ and }\\
\epsilon(u_1,v_1,u_2,v_2) &= \epsilon_1(u_1,v_1,u_2,v_2) \wedge \epsilon_2(u_1,v_1,u_2,v_2).
\end{align*}
We consider the set
$$L(u,v)=\{(u_2,v_2) \in T\;|\; \mathcal M \models \epsilon(u,\st(\delta(u,v)),u_2,v_2)\}$$ for any $(u,v) \in T$.
The set $L(u,v)$ is an $\mathcal M$-definable set for any fixed $(u,v) \in T$.
We consider two cases separately.
\medskip

\textbf{Case A.} There exists $(\widetilde{u}, \widetilde{v}) \in T$ such that $\dim L(\widetilde{u}, \widetilde{v})<d$.
\medskip

Fix $(\widetilde{u}, \widetilde{v}) \in T$ such that $\dim L(\widetilde{u}, \widetilde{v})<d$.
We consider the following three sets:
\begin{align*}
\mathfrak T_1(T,C,f) &=\{(u,v) \in T\;|\; \mathcal N \models ((u,v) \in L(\widetilde{u}, \widetilde{v}) \wedge f(u,v,\alpha)<\beta)\},\\
\mathfrak T_2(T,C,f) &=\{(u,v) \in T\;|\; \mathcal N \models (\neg \epsilon_1(\widetilde{u}, \st(\delta(\widetilde{u}, \widetilde{v})),u,v) \wedge f(u,v,\alpha)<\beta)\},\\
\mathfrak T_3(T,C,f) &=\{(u,v) \in T\;|\; \mathcal N \models (\neg \epsilon_2(\widetilde{u}, \st(\delta(\widetilde{u}, \widetilde{v})),u,v) \wedge f(u,v,\alpha)<\beta)\}.
\end{align*}
We obviously have $\mathfrak T(T,C,f)=\bigcup_{i=1}^3 \mathfrak T_i(T,C,f)$.
We have only to show that $\mathfrak T_i(T,C,f)$ is $\mathcal M$-definable for each $1 \leq i \leq 3$.
Since $\dim L(\widetilde{u}, \widetilde{v})<d$ by our case hypothesis, the set $\mathfrak T_1(T,C,f)$ is $\mathcal M$-definable by the induction hypothesis.

We next consider $\mathfrak T_2(T,C,f)$.
Fix an arbitrary $(u,v) \in T$ with $$\mathcal N \models  \neg \epsilon_1(\widetilde{u}, \st(\delta(\widetilde{u}, \widetilde{v})),u,v).$$
Note that the equalities $r_i(\widetilde{u}, \st (\delta(\widetilde{u}, \widetilde{v})))=r_i(\widetilde{u}, \widetilde{v})$ hold true for $i=1,2$ by the definition of $\delta$.
There exists $v' \in M$ such that $r_1(\widetilde{u}, \widetilde{v})<v'<r_2(\widetilde{u}, \widetilde{v})$, $v'<\st(\delta(\widetilde{u}, \widetilde{v}))$ and $f(u,v,\alpha) < f(\widetilde{u},v',\alpha)$.
We have $v'<\delta(\widetilde{u}, \widetilde{v})$ because $v' \in M$.
Since $f(u, \cdot, \alpha)$ is strictly increasing on $(r_1(\widetilde{u}, \widetilde{v}),r_2(\widetilde{u}, \widetilde{v}))$, we get $f(\widetilde{u},v',\alpha)<f(\widetilde{u},\delta(\widetilde{u}, \widetilde{v}),\alpha)=\beta$.
We finally obtain $f(u,v,\alpha)<\beta$.
Therefore, we get $\mathfrak T_2(T,C,f)=\{(u,v) \in T\;|\; \mathcal N \models \neg \epsilon_1(\widetilde{u}, \st(\delta(\widetilde{u}, \widetilde{v})),u,v)\}$, which is $\mathcal M$-definable by the induction hypothesis.
We can prove that, if $\mathcal N \models  \neg \epsilon_2(\widetilde{u}, \st(\delta(\widetilde{u}, \widetilde{v})),u,v)$, we have $f(u,v,\alpha)>\beta$ in the same manner.
The set $\mathfrak T_3(T,C,f)$ is an empty set.
We have demonstrated Lemma \ref{lem:main1} in Case A.
\medskip

\textbf{Case B.} The equality $\dim L(\widetilde{u}, \widetilde{v})=d$ holds true for each $(\widetilde{u}, \widetilde{v}) \in T$.
\medskip

We consider the $\mathcal M$-definable set $\Lambda$ defined by 
\begin{align*}
\Lambda=\{(u,v,v') \in T \times M\;|\; r_1(u,v)<v'<r_2(u,v) \wedge \\
\qquad\dim(\{(u_2,v_2) \in T\;|\; \mathcal M \models \epsilon(u,v',u_2,v_2)\}=d)\}.
\end{align*}
It is a definable set by the definition of dimension.
Our case hypothesis implies that $$(u,v,\st(\delta(u,v))) \in \Lambda$$ for each $(u,v) \in T$.
We demonstrate that the fiber $\Lambda_{(u,v)}$ of $\Lambda$ at $(u,v) \in T$ is of dimension zero.
Fix $(u,v) \in T$ for a while.
Consider the $\mathcal M$-definable set $$\mathfrak Z=\{(u_2,v_2,v') \in T \times M\;|\; \mathcal M \models  \epsilon(u,v',u_2,v_2) \wedge ((u,v,v') \in \Lambda)\}.$$
Let $q_1:M^{m+1} \rightarrow M^m$ and $q_2:M^{m+1} \rightarrow M$ be the coordinate projections onto first $m$ coordinates and onto the last coordinate, respectively.
The projection image $q_2(\mathfrak Z)$ coincides with $\Lambda_{(u,v)}$ and the fiber $\mathfrak Z \cap q_2^{-1}(v')$ is of dimension $d$ for any $v' \in \Lambda_{(u,v)}$ by the definition of the set $\Lambda$.
We get $$\dim \mathfrak Z = \dim \Lambda_{(u,v)} + d$$ by Proposition \ref{prop:dim}(8).
On the other hand, for any $(u_2,v_2,v') \in \mathfrak Z$, we obtain $\mathcal M \models \epsilon(u,v',u_2,v_2)$.
It implies that, for all $r_1(u,v')<v''<r_2(u,v')$, we get $f(u,v'',\alpha) \leq f(u_2,v_2,\alpha)$ if $v''<v'$ and $f(u,v'',\alpha) \geq f(u_2,v_2,\alpha)$ if $v''>v'$.
Since $f(u,\cdot,\alpha)$ is continuous on $(r_1(u,v'),r_2(u,v'))$, the equality $$f(u,v',\alpha)=f(u_2,v_2,\alpha)$$ holds true.
When we fix $(u_2,v_2) \in q_1(\mathfrak Z)$, at most one $v'$ satisfies the above equality because $f(u,\cdot,\alpha)$ is strictly increasing.
It means that the fiber $q_1^{-1}(u_2,v_2) \cap \mathfrak Z$ is a singleton.
We therefore get $$\dim \mathfrak Z = \dim q_1(\mathfrak Z) \leq \dim T = d$$ by Proposition \ref{prop:dim}(8).
We have demonstrated $\dim \Lambda_{(u,v)}=0$.
In particular, the fiber $\Lambda_{(u,v)}$ is discrete and closed by Proposition \ref{prop:dim}(3).

The point $\st(\delta(u,v))$ is the closest point in the $\mathcal M$-definable subset $$\{v' \in M\;|\; (u,v,v') \in \Lambda\}$$ of $M$ to $\delta(u,v)$ because $\st(\delta(u,v)) \in \Lambda_{(u,v)}$ and $\Lambda_{(u,v)}$ is discrete and closed for any $(u,v) \in T'$.
In other word, the set 
\begin{align*}
\Gamma &= \{(u,v,v') \in (T' \times M) \cap \Lambda\;|\; \mathcal N \models \forall v''\ ((u,v,v'') \in \Lambda\\ &\qquad \rightarrow |\delta(u,v)-v'| \leq |\delta(u,v)-v''|)\}\\
 &=\{(u,v,v') \in (T' \times M) \cap \Lambda\;|\; \mathcal N \models \forall v''\  \forall w
 \ (\Theta(u,v,w,\alpha,\beta) \\
 &\qquad \wedge (u,v,v'') \in \Lambda) \rightarrow (|w-v'| \leq |w-v''|)\}
\end{align*}
is the graph of the composition $\st \circ \delta$.
The definition of the formula $\Theta(u,v,v',w,z)$ is found in the equality (\ref{eqeq:theta}).
It is $\mathcal M$-definable by the induction hypothesis because $\dim (T' \times M) \cap \Lambda \leq \dim T'<d$ by Proposition \ref{prop:dim}(8).
We have demonstrated that the composition $\st \circ \delta$ is $\mathcal M$-definable.
\medskip

We consider the set 
\begin{align*}
Q_1&=\{(u,v) \in T'\;|\;\st(\delta(u,v))<\delta(u,v)\}\\
&=\{(u,v) \in T'\;|\; \mathcal N \models \forall v',\ \Theta(u,v,v',\alpha,\beta) \wedge (\st(\delta(u,v))<v')\}.
\end{align*}
It is $\mathcal M$-definable because Theorem \ref{thm:main1} holds true for $T'$.
The set 
\begin{align*}
P_1=\{(u,v) \in T\;|\; \st(\delta(\rho(u,v)))<\delta(\rho(u,v))\}=\rho^{-1}(Q_1)
\end{align*}
is also $\mathcal M$-definable.
The set 
\begin{align*}
P_2=\{(u,v) \in T\;|\; \st(\delta(\rho(u,v))) \geq \delta(\rho(u,v))\}
\end{align*}
is also $\mathcal M$-definable for the same reason.
We then have
\begin{align*}
\mathfrak T(T,C,f)&=\{(u,v) \in T\;|\; f(u,v,\alpha)<\beta\}\\
&=\{(u,v) \in P_1\;|\; r_1(u,v)<v \leq \st(\delta(\rho(u,v)))\} \\
&\qquad \cup \{(u,v) \in P_2\;|\; r_1(u,v)<v < \st(\delta(\rho(u,v)))\}
\end{align*}
because $f(u,\cdot,\alpha)$ is strictly increasing on $(r_1(u,v),r_2(u,v))$.
Therefore, $\mathfrak T(T,C,f)$ is $\mathcal M$-definable because the composition $\st \circ \delta$ is $\mathcal M$-definable.
We have demonstrated Lemma \ref{lem:main1}.
\end{proof}

\section{Corollaries of the main theorem}\label{sec:corollaries}
Using Theorem \ref{thm:main1}, we can get the following corollaries in the same manner as \cite{vdD2}.
Let $\dim_{\mathcal M}S$ denote the dimension of an $\mathcal M$-definable set $S$.
We also define $\dim_{\mathcal N}S$ in the same manner.

\begin{corollary}\label{cor:main1_1}
Let $\mathcal L$, $\mathcal M$ and $\mathcal N$ be as in Theorem \ref{thm:main1}.
For any $\mathcal N$-definable subset $S$ of $N^n$ parameterized by $\mathcal M$-bounded parameters, we have $\dim_{\mathcal N}S \geq \dim_{\mathcal M} S \cap M^n.$
\end{corollary}
\begin{proof}
We prove the lemma by induction on $(n,k=\dim_{\mathcal N}(S))$ under the lexicographic order.
Set $T=S \cap M^n$.
When $k=0$, the set $S$ is discrete and closed by Proposition \ref{prop:dim}(3).
Since $T$ is a subset of $S$, it is discrete or an empty set.
Therefore, we have $\dim_{\mathcal M}(T) \leq 0$ by Proposition \ref{prop:dim}(3).

The lemma is obvious when $n=k$.
The lemma has been demonstrated when $n=1$.

We consider the case in which $k>0$.
Let $\mathcal P_k$ be the set of all coordinate projections from $N^n$ onto $N^k$.
It is a finite set and we fix a linear order on the set $\mathcal P_k$.
The image of $S$ under some coordinate projection in $\mathcal P_k$ has a nonempty interior by the definition of dimension.
Let $\pi_S:N^n \rightarrow N^k$ be the largest element in $\mathcal P_k$ under which  the image of $S$ has a nonempty interior.
The same notation $\pi_S$ also denotes the coordinate projection $M^n \rightarrow M^k$.
This abuse of notations will not confuse the readers.
Set $S_1=\{x \in S\;|\; \dim_{\mathcal N}(S \cap \pi_S^{-1}(\pi_S(x)))>0\}$ and $T_1=S_1 \cap M^n$.
Either $S_1$ is of dimension smaller than $k$ or $\pi_{S_1}$ is smaller than $\pi_{S}$ in $\mathcal P_k$.
We have $\dim_{\mathcal N} S_1 \geq \dim_{\mathcal M} T_1$ by the induction hypothesis.

Set $S_2= S \setminus S_1$.
By the definition of $S_2$, we have $\dim_{\mathcal N} S_2 \cap \pi_S^{-1}(\pi_S(x)) =0$ for all $x \in S_2$ by the definition of $S_2$.
Set $T_2=S_2 \cap M^n$.
We immediately get $\dim_{\mathcal M} T_2 \cap \pi_S^{-1}(\pi_S(x))  \leq 0$ for all $x \in S_2$ for the same reason as the case in which $k=0$.
Since $\pi_S(T_2) \subseteq \pi_S(S_2) \cap M^{n-1}$ and $\dim_{\mathcal N} \pi_S(S_2)=\dim_{\mathcal N} S_2 \leq k$ by Proposition \ref{prop:dim}(8), we get $\dim_{\mathcal M} \pi_S(T_2)  \leq \dim_{\mathcal M} \pi_S(S_2) \cap M^{n-1} \leq k$ by the induction hypothesis.
We get $\dim_{\mathcal M} T_2\leq k$ by Proposition \ref{prop:dim}(8).
We immediately get $\dim_{\mathcal M}T =\max \{\dim_{\mathcal M}T_1,\dim_{\mathcal M}T_2\} \leq k$ by Proposition \ref{prop:dim}(4).
\end{proof}

\begin{remark}
The inequality in Corollary \ref{cor:main1_1} may be strict.
For instance, consider a singleton $S$ defined by an $\mathcal M$-bounded element in $N \setminus M$.
We obviously have $S \cap M=\emptyset$ and $0=\dim_{\mathcal N}S > \dim_{\mathcal M}(M \cap S)=-\infty$.
\end{remark}

\begin{corollary}\label{cor:main1_2}
Let $\mathcal L$, $\mathcal M$ and $\mathcal N$ be as in Theorem \ref{thm:main1}.
Let $\mathcal V$ be the set of $\mathcal M$-bounded elements in $N$.
Consider an $\mathcal N$-definable subset $S$ of $N^n$ parameterized by $\mathcal M$-bounded parameters.
The set $\st(S \cap \mathcal V^n)$ is $\mathcal M$-definable.
\end{corollary}
\begin{proof}
Let $x$, $y$ and $z$ denote $n$-tuples of elements in a set, and the notations $x_i$, $y_i$ and $z_i$ denote the $i$-th element, respectively.
Consider the set $T=\{(x,y) \in M^n \times M^n\;|\; \mathcal N \models \exists z \in S,\ x_i<z_i<y_i\text{ for all } 1 \leq i \leq n\}$.
It is $\mathcal M$-definable by Theorem \ref{thm:main1}.
We obviously have
\begin{align*}
\st(S \cap \mathcal V^n) &=\{z \in M^n\;|\; \mathcal M \models (\forall x_1,\ldots, \forall x_n,\forall y_1,\ldots, \forall y_n, \\
&(x_1<z_1<y_1)\wedge \ldots\wedge(x_n<z_n<y_n) \rightarrow (x,y) \in T)\}.
\end{align*} 
It means that $\st(S \cap \mathcal V^n)$ is $\mathcal M$-definable.
\end{proof}

\begin{corollary}\label{cor:main1_3}
Let $\mathcal L$, $\mathcal M$ and $\mathcal N$ be as in Theorem \ref{thm:main1}.
Consider a $\mathcal N$-definable function $f:N^n \rightarrow N$ parameterized by $\mathcal M$-bounded parameters.
The three sets
\begin{align*}
&D_{-\infty} =\{x \in M^n\;|\; f(x) < y \ (\forall y \in M)\},\\
&D_{\infty} =\{x \in M^n\;|\; f(x) > y \ (\forall y \in M)\} \text{ and }\\
&D=M^n \setminus (D_{-\infty} \cup D_{\infty})
\end{align*}
and the map $g:D \rightarrow M$ given by $x \mapsto \st(f(x))$ are all $\mathcal M$-definable.
\end{corollary}
\begin{proof}
Consider the set $X=\{(x,y) \in M^n \times M\;|\; f(x)<y\}$.
It is $\mathcal M$-definable by Theorem \ref{thm:main1}.
We have $D_{-\infty}=\{x \in M^n\;|\; \forall y\  (x,y) \in X\}$.
It is obviously $\mathcal M$-definable.
We can show that $D_{\infty}$ is $\mathcal M$-definable, similarly.
The $\mathcal M$-definability of $D$ is now trivial.

We next consider the set $Y=\{(x,y_1,y_2) \in M^n \times M \times M\;|\; y_1<f(x)<y_2\}$.
It is $\mathcal M$-definable by Theorem \ref{thm:main1}.
The graph of $g$ is given by
$\{(x,y) \in M^n \times M\;|\; \mathcal M \models \forall y_1, \forall y_2, y_1<y<y_2 \rightarrow (x,y_1,y_2) \in Y\}$.
We have shown that $g$ is $\mathcal M$-definable.
\end{proof}

\end{document}